\documentclass[12pt]{amsart}
\usepackage[osf,sc]{mathpazo}
\usepackage{amssymb}

\usepackage{geometry}
\usepackage{bm}
\usepackage{graphicx}
\usepackage{hyperref}
\hypersetup{
	colorlinks=true, 
	linktoc=all,     
	linkcolor=blue,
	citecolor=red,
	filecolor=black,
	urlcolor=blue	
}
\usepackage{enumerate}
\usepackage[inline]{enumitem}
\makeatletter
\newcommand{\inlineitem}[1][]{%
	\ifnum\enit@type=\tw@
	{\descriptionlabel{#1}}
	\hspace{\labelsep}%
	\else
	\ifnum\enit@type=\z@
	\refstepcounter{\@listctr}\fi
	\quad\@itemlabel\hspace{\labelsep}%
	\fi} \makeatother
\newcommand{\gd}{\delta}

\newcommand{\gp}{\pi}

\newcommand{\gs}{\sigma}

\newcommand{\Gd}{\Delta}

\newcommand{\Gs}{\Sigma}

\newcommand{\Gom}{\Omega}

\newcommand{\subs}{\subset}

\newcommand{\sbq}{\subseteq}

\newcommand{\bs}{\backslash}

\newcommand{\nin}{\notin}

\newcommand{\mbb}{\mathbb}

\newcommand{\mcl}{\mathcal}

\newcommand{\us}{\underset}
\newcommand{\os}{\overset}

\newcommand{\lra}{\longrightarrow}

\newcommand{\Z}{\mbb Z}

\newcommand{\La}{\Leftarrow}

\newcommand{\es}{\emptyset}
\newcommand{\ora}{\overrightarrow}

\newcommand{\equ}[1]{%
	\begin{equation*}
		#1
	\end{equation*}
}
\newcommand{\equa}[1]{%
	\begin{equation*}
		\begin{aligned}
			#1
		\end{aligned}
	\end{equation*}
}

\DeclareMathOperator{\Det}{Det}
\DeclareMathOperator{\Hom}{Hom}

\newtheorem{theorem}{Theorem}[section]

\newtheorem{prop}[theorem]{Proposition}
\newtheorem{lemma}[theorem]{Lemma}

\theoremstyle{definition}
\newtheorem{defn}[theorem]{Definition}
\newtheorem{example}[theorem]{Example}

\theoremstyle{remark}

\newtheorem{remark}[theorem]{Remark}
\numberwithin{equation}{section}
\makeatletter
\def\namedlabel#1#2{\begingroup
	\def\@currentlabel{#2}%
	\label{#1}\endgroup
}
\makeatother

\newtheorem*{thmOmega}{\bf{Theorem} $\bm{\Gom}$}

\begin{document}
\title[On the Reduciblity of a Certain Type of Matroid by a Point]{On the Reduciblity of a Certain Type of Rank $3$ Uniform Oriented Matroid by a Point}
\author[C.P. Anil Kumar]{Author: C.P. Anil Kumar*}
\address{Post Doctoral Fellow in Mathematics, Room No. 223, I Floor, Main Building, Harish-Chandra Research Institute, (Department of Atomic Energy, Government of India),
Chhatnag Road, Jhunsi, Prayagraj (Allahabad)-211019, Uttar Pradesh, INDIA
}
\email{akcp1728@gmail.com}
\thanks{*The work is done when the author is a Post Doctoral Fellow at HRI, Allahabad.}
\begin{abstract}
For a positive integer $n\geq 3$, the sides and diagonals of a convex $n\operatorname{-}$gon divide the interior of the convex $n\operatorname{-}$gon into finitely (polynomial in $n$) many regions bounded by them. In this article,  we associate to every region a unique $n\operatorname{-}$cycle in the symmetric group $S_n$ of a certain type (defined as $2$-standard consecutive cycle) by studying point arrangements in the plane. Then we find that there are more (exponential in $n$) number of such cycles leading to the conclusion that not every region labelled by a cycle appears in every convex $n\operatorname{-}$gon. In fact most of them do not occur in any given single convex $n\operatorname{-}$gon. Later in the main theorem of this article we characterize combinatorially those cycles (defined as definite cycles) whose corresponding regions occur in every convex $n\operatorname{-}$gon and those cycles (defined as indefinite cycles) whose corresponding regions do not occur in every convex $n\operatorname{-}$gon. As a consequence we characterize those one point extensions of a uniform rank $3$ convex oriented matroid for which the one point extension is reducible by the, one point, when it lies inside the convex hull. 
\end{abstract}
\maketitle
\section{\bf{Introduction}}
The regions of a convex $n\operatorname{-}$gon when divided by the diagonals has been studied in various contexts in the literature (J.W.~Freeman~\cite{MR0412967}, R.~Honsberger~\cite{MR0419117} [Chapter $9$], J.~Herman, R.~Kucera, J.~Simsa~\cite{MR1950450} [Chapter $3$], the OEIS sequence A006522~\cite{OEIS}). The problems on points in general position are treated in Chapter $8$ of  P.~Brass, W.~O.~J.~Moser, J.~Pach~\cite{MR2163782}. Here, in~\cite{MR2163782}, two convex $7$-gons with different diagonal arrangements is given. In this article, we consider certain point arrangements in the plane $\mbb{R}^2$ or certain vector configurations in the space $\mbb{R}^3$, that is, acyclic oriented uniform matroids of rank $3$, where $n$ points are in a convex position in the affine plane $\mbb{R}^2$, with an additional point $p$ contained in the interior of a region of the convex hull of the $n$ points when the hull is divided by its diagonals. The region becomes the {\it residence} of the point $p$, (see L.~M.~Kelly and W.~O.~J.~Moser~\cite{MR0097014}, page 211 for definition of ``{\it residence}"). To study the residence of $p$, we associate, to any region of the convex $n\operatorname{-}$gon when divided by diagonals, a unique 2-standard consecutive cycle. 
The method of associating cycle invariants as a combinatorial model to point arrangements in the plane has already been explored by J.~E.~Goodman and R.~Pollack~\cite{MR0583961},~\cite{MR0769218}.
A similar method is explained in~\cite{MR1723092} [Chapter $10$]. 
Now we mention a couple of definitions to proceed further.
\begin{defn}[Reducibility of an oriented matroid by a point]
Let $M$ denote a rank $d$ oriented matroid on $E=\{1,2,\cdots,n\}$, and let $\widetilde{\mcl{R}(M)}$ denote the set of real $n\times d$ matrices realizing $M$. The realization space $\mcl{R}(M)$ is defined as the quotient of the topological space $\widetilde{\mcl{R}(M)} \subseteq \mbb{R}^{nd}$ by the canonical action of the group $GL(d,\mbb{R})$ of nonsingular $d\times d$ matrices. 
Deleting a point $e\in E$ in $M$ induces a natural map $\gp_e: \mcl{R}(M)\lra \mcl{R}(M\bs e)$. Each fiber of this map is a convex set, namely it is the {\it residence} of the point $e$ in a realization of $M$. We say the oriented matroid $M$ is reducible by a point $e\in E$ provided the deletion map $\gp_e: \mcl{R}(M)\lra \mcl{R}(M\bs e)$ is surjective. In other words, $M$ is reducible by $e$ if every real realization of the minor $M\bs e$ extends to a realization of $M$.
\end{defn}
\begin{defn}[The convex matroid $C_n$ and a one-point extension $D_n$ for $n\geq 3$]
\label{defn:ConvexMatroid}
Let $n\geq 3$ be a positive integer.
Let $C_n$ be the rank 3 uniform oriented matroid consisting of a set $E_n$ of $n$ points where in any realization of $C_n$ there are $n$ points in convex position in the affine plane $\mbb{R}^2$, that is, if $E_n=\{1,2,\cdots,n\}$ then for any $1\leq i<j<k<l\leq n$ and circuit $X=\{i<j<k<l\}\subseteq  E_n$ we have $X^{+}=\{i,k\},X^{-}=\{j,l\}$ or vice-versa with $X=X^+\sqcup X^-$ as the signed circuit. So in a realization matrix $A=(v_1,v_2,\cdots,v_n)\in Mat_{3,n}(\mbb{R})$ of $C_n$ we have $\Det(v_i,v_j,v_k)>0$ for every $1\leq i<j<k\leq l$. It can be shown that this matroid $C_n$ is acyclic, that is, there exists a linear functional $l_A\in \Hom(\mbb{R}^3,\mbb{R})$ such that $l_A(v_i)>0$ for all $1\leq i\leq n$ for any such realization matrix $A$.

Let $D_n=C_n\cup \{n+1\}$ denote the one point extension of the matroid $C_n$ consisting of an additional point $p=n+1$ such that in any realization, the point $p$ is in general position contained in the interior of the convex hull of $n$ points of $C_n$. So if $B=(v_1,v_2,\cdots,v_n,v_{n+1})\in Mat_{3,n+1}(\mbb{R})$ is any realization of $D_n$ then $A=(v_1,v_2,\cdots,v_n)$ is a realization of $C_n$ and we have $\Det(v_i,v_{i+1},v_{n+1})>0$ for any $1\leq i\leq n-1$ and $\Det(v_n,v_1,v_{n+1})>0$.  	
\end{defn}

\begin{defn}[Definite Region and Indefinite Region]
\label{defn:DefiniteRegion}
We say a region/residence $R$ of a point $p$ in a convex $n$-gon when divided by diagonals is a reducible region or definite region if the corresponding matroid $D_n$ is reducible by the point $p$. We say it is an indefinite region if $D_n$ is not reducible by the point $p\in R$.  	
\end{defn}

In J.~Richter and B.~Sturmfels~\cite{MR0994170}, the authors introduce the notion of isolated points for uniform oriented matroids and study the concept of reducibility along with it. In fact they prove in ~\cite{MR0994170} Theorem 3.4, page 397 that, in a rank 3 uniform oriented matroid $M$, any point $p\in M$ is either isolated or $M$ is reducible by $p$. Using this they prove the main result that, if $M$ is a uniform rank 3 oriented matroid with at most eight points then the realization space $\mcl{R}(M)$ of $M$ is contractible, and hence $M$ satisfies the isotopy property (which states that the space $\mcl{R}(M)$ is path-connected). 
For an oriented matroid $M$, an ordering $(e_1,e_2,\cdots,e_n)\in E$ is called a reduction sequence for $M$ if either $d=n$ or $M$ is reducible by $e_n$ and $(e_1,e_2,\cdots,e_{n-1})$ is a reduction sequence for $M\bs e_n$. Now the following holds:

\begin{prop}
\label{prop:OM}
Let $M$ be an oriented matroid. 
\begin{enumerate}
\item (Lemma 2.1~\cite{MR0994170}) Suppose $M$ is reducible by $e$ then $\mcl{R}(M)$ is homotopy equivalent to $\mcl{R}(M\bs e)$. 
\item (Corollary 2.2~\cite{MR0994170}) If $M$ admits a reduction sequence then its realization space $\mcl{R}(M)$ is contractible.
\item (Corollary 3.5~\cite{MR0994170}) If $M$ has six or less points, then $M$ is reducible by each of its points.
\end{enumerate}
\end{prop}

Proposition~\ref{prop:OM} can be applied to the matroid $D_n$ in main Theorem~\ref{theorem:DefIndefRegions} of the article giving rise to some more consequences. 

\subsection{Generic Realization Spaces}
Let $M$ be an oriented matroid of rank $d$ on $E=\{1,2,\cdots,n\}$. Let $\mcl{H}=\{H\mid H\text{ is a hyperplane of the matroid }M, \text{that is, a flat of}$    $\text{rank }d-1\}$. Then we have $\mcl{C}^*=\{E\bs H\mid H\in \mcl{H}\}$ is the set of cocircuits and in an oriented matroid we have the decomposition $E\bs H=(E\bs H)^+\sqcup(E\bs H)^-$ as a signed cocircuit representing the two sides of the hyperplane $H$. In a realization matrix $A\in Mat_{d,n}(\mbb{R})$ of the matroid $M$, let $H_A$ denote the realization hyperplane (a subspace of $\mbb{R}^d$ of dimension $d-1$) corresponding to a hyperplane $H$ of the matroid $M$. 
Let $\widetilde{\mcl{R}_{gen}(M)}$ denote the set of real $(n\times d)$-matrices $A\in Mat_{d,n}(\mbb{R})$ realizing $M$ for which the hyperplanes of the matroid $M$ form a hyperplane arrangement which satisfy the following intersection property.

If $H_1,H_2,\cdots H_r$ are $r$ distinct hyperplanes of the matroid and $H_1\cap H_2\cap \cdots \cap H_r$ is a flat of rank $k\geq 0$ then we must have $(H_1)_A\cap (H_2)_A\cap \cdots (H_r)_A$ is of vector subspace in $\mbb{R}^d$ of dimension $\max(k,d-r)$.
The generic realization space $\mcl{R}_{gen}(M)$ is defined as the quotient of the topological space $\widetilde{\mcl{R}_{gen}(M)} \subseteq \widetilde{\mcl{R}(M)} \subseteq \mbb{R}^{nd}$ by the canonical action of the group $GL(d,\mbb{R})$ of nonsingular $d\times d$ matrices.

\begin{example}
Let $C_n$ be the convex matroid given in Definition~\ref{defn:ConvexMatroid}. Then $\mcl{R}_{gen}(M)$ consists of those realizations of $C_n$ where the diagonals of the convex $n$-gon have the property that, no three diagonals have a common point of intersection.   
\end{example}

We find that for any $n$, there are $2^{n-1}-n$ (exponential in $n$) $2$-standard consecutive cycles where as the number of regions in a convex $n\operatorname{-}$gon in which no three diagonals are concurrent at an interior point of the polygon is given by $\frac{(n-1)(n-2)(n^2-3n+12)}{24}$ which is a polynomial in $n$. Hence we conclude that not every region occurs in every convex $n\operatorname{-}$gon. This leads to definite  cycles/regions (those which occur always) and indefinite cycles/regions (those which do not occur always). We see that definite regions are residences of a point $p=n+1$ such that the matroid $D_n$ is reducible by $p$ and indefinite regions are those such that the matroid $D_n$ is not reducible by $p$. It happens that, for $n\leq 5$ all regions are definite. This agrees with Proposition~\ref{prop:OM}(3) because $n+1\leq 6$. From $n=6$ the indefinite regions start to appear. However for $n\geq 8$ both the types can be combinatorial characterized.
Main Theorem~\ref{theorem:DefIndefRegions} and Theorem~\ref{theorem:Reducibility} in this article characterizes those regions or residences of the point $p$ for which, the matroid $D_n$ is reducible by the point $p$. Main Theorem~\ref{theorem:DefIndefRegions} is stated in Section~\ref{sec:DefIndefRegions} after the required definitions and motivation for these definitions. Theorem~\ref{theorem:Reducibility} is stated in the last section. 

\section{\bf{The Number of Regions in a Convex $n\operatorname{-}$gon with Generic Diagonals}}
In this section, we consider the regions of a convex $n\operatorname{-}$gon which has generic diagonals and compute the number of regions formed by the diagonals and sides.
\begin{defn}[Point Arrangement in the Plane]
	Let $\mcl{P}_n=\{P_1,P_2,\ldots,$ $P_n\}$ be a finite set of $n$ points in the plane $\mbb{R}^2$. Then $\mcl{P}_n$ is said to be a point arrangement in the plane if no three points are collinear.
\end{defn}
\begin{defn}[Side and Diagonal of a Convex $n\operatorname{-}$gon]
Let $\mcl{P}_n=\{P_1,\ldots,P_n\}$ be a point arrangements in $\mbb{R}^2$ such that the points form a convex $n\operatorname{-}$gon in the anticlockwise manner $P_1\lra P_2\lra \ldots \lra P_n\lra P_1$. A side of the convex $n\operatorname{-}$gon is a line $P_iP_j$ with $j\equiv i\pm 1 \mod n$ where $1\leq i\neq j\leq n$. A diagonal of the convex $n\operatorname{-}$gon is a line $P_iP_j$ with $j\not\equiv i\pm 1\mod n$ where $1\leq i\neq j\leq n$. 
\end{defn}
\begin{defn}[Convex $n\operatorname{-}$gon with Generic Diagonals]
Let $\mcl{P}_n=\{P_1,\ldots,P_n\}$ be a point arrangement in $\mbb{R}^2$ such that the points form a convex $n\operatorname{-}$gon in the anticlockwise manner $P_1\lra P_2\lra \ldots \lra P_n\lra P_1$. We say the convex $n\operatorname{-}$gon has generic diagonals if for every three pairs of subscripts $\{i_t,j_t\},1\leq i_t\neq j_t\leq n,1\leq t\leq 3$ with
\equ{\{i_1,j_1\}\cap\{i_2,j_2\}\cap\{i_3,j_3\}=\es,}
$P_{i_1}P_{j_1},P_{i_2}P_{j_2},P_{i_3}P_{j_3}$ do not concur in the plane $\mbb{R}^2$.
\end{defn}
\begin{remark}[A Region]
A region is defined to be a connected component of the interior of the convex $n\operatorname{-}$gon when the diagonals and sides are removed.
\end{remark}
Now we state the theorem.
\begin{theorem}[Chapter $9$, pp. 99-107 in R.~Honsberger~\cite{MR0419117}, J.~W.~Freeman~\cite{MR0412967}]
\label{theorem:NumberofRegions}
Let $\mcl{P}_n=\{P_1,\ldots,P_n\}$ be a point arrangement in $\mbb{R}^2$ such that the points form a convex $n\operatorname{-}$gon in the anticlockwise manner $P_1\lra P_2\lra \ldots \lra P_n\lra P_1$ and which has generic diagonals. Then the number of regions formed by the convex $n\operatorname{-}$gon is given by 
\equ{\frac{(n-1)(n-2)(n^2-3n+12)}{24}.}
\end{theorem}
\section{\bf{On the $2\operatorname{-}$standard consecutive structure of an $n\operatorname{-}$cycle and the Regions of a Convex $n\operatorname{-}$gon}}
In this section, for a positive integer $i$, we define $i\operatorname{-}$standard consecutive structure on an $n\operatorname{-}$cycle. We associate to a region in a convex $n\operatorname{-}$gon formed by the diagonals, a $2\operatorname{-}$standard consecutive cycle. 

Now we introduce a structure on a permutation as follows.
\begin{defn}[$i\operatorname{-}$Standard Cycle]
	We say an $n\operatorname{-}$cycle $(1=a_1a_2\ldots a_n)$ is an $i\operatorname{-}$standard cycle if there exists a way to write
	the integers $a_i,i=1,\ldots,n$ as $i$ sequences of inequalities as follows:
	\equa{& a_{11}<a_{12}<\ldots<a_{1j_1}\\
		& a_{21}<a_{22}<\ldots<a_{2j_2}\\
		& a_{31}<a_{32}<\ldots<a_{3j_3}\\
		&\vdots \\
		&a_{i1}<a_{i2}<\ldots<a_{ij_i}}
	where $\{a_{st}\mid 1\leq s\leq i,1\leq t\leq j_s\}=\{a_1,a_2,\ldots,a_n\}=\{1,2,\ldots,n\},
	j_1+j_2+\ldots+j_i=n$ and $i$ is minimal, that is, there exists no smaller integer with such property
	and furthermore that $a_{s(t+1)}$ occurs to the right of $a_{st}$ for every $1\leq s\leq i$ and
	$1\leq t \leq j_s-1$ in this cycle arrangement $(a_1=1,a_2,\ldots,a_n)$.
\end{defn}
Define $i\operatorname{-}$standard consecutive structure on an $n\operatorname{-}$cycle as follows.
\begin{defn}[$i\operatorname{-}$Standard Consecutive Cycle]
	\label{defn:2StandardConsecutiveStructure}
	We say an $n\operatorname{-}$cycle $(1=a_1a_2\ldots a_n)$ is a consecutive $i\operatorname{-}$standard cycle 
	or a $i\operatorname{-}$standard consecutive cycle if we have
	\equ{a_{s1}<a_{s2}<\ldots<a_{sj_s}}
	and in addition $a_{st}=a_{s1}+(t-1),1\leq t\leq j_s,1\leq s\leq i$
	where $\{a_{st}\mid 1\leq s\leq i,1\leq t\leq j_s\}=\{a_1,a_2,\ldots,a_n\}=\{1,2,\ldots,n\},
	j_1+\ldots+j_s=n$ and $a_{s(t+1)}$ occurs to the right of $a_{st}$ for every $s=1,\ldots,i$ and
	$1\leq t \leq j_s-1$ in this cycle arrangement $(1=a_1a_2\ldots a_n)$
	and $i$ is minimal, that is, there exists no smaller integer with such property.
	If the minimal value of $i$ is two then we say that the cycle has $2\operatorname{-}$standard consecutive structure.
\end{defn}
\begin{example}
	The $5\operatorname{-}$cycle $(14523)$ it is a $2\operatorname{-}$standard consecutive cycle. 
	However it has the following two $2\operatorname{-}$standard structures.
	\begin{itemize}
		\item $1<4<5,2<3$ (not consecutive).
		\item $1<2<3,4<5$ (consecutive).
	\end{itemize}
\end{example}

\begin{defn}[Line Cycle at a Point]
	Let $\mcl{P}_n=\{P_1,P_2,\ldots,P_n\}$ be a point arrangement in the plane. Fix $P_i \in \mcl{P}_n$ for some $1\leq i\leq n$. Consider the $(n-1)$ lines $L^i_j$ joining $P_i,P_j$ for $1\leq j\leq n,j\neq i$. An anticlockwise traversal around the point $P_i$ cuts the lines $L^i_j$ in a cycle $\gs_i\in S_{n-1}$, a symmetric group over the elements $\{1,2,\ldots,i-1,i+1,\ldots,n\}$. This cycle is defined to be the line cycle at $P_i$ for the arrangement $\mcl{P}_n$.
\end{defn}
The following result enables us to identify regions by their $2$-standard consecutive cycles.
\begin{theorem}
	\label{theorem:LineCycleofaRegion}
Let $\mcl{P}_n=\{P_1,\ldots,P_n\}$ be a point arrangement in $\mbb{R}^2$ such that the points form a convex $n\operatorname{-}$gon in the anticlockwise manner $P_1\lra P_2\lra \ldots \lra P_n\lra P_1$ and which has generic diagonals. Let $R$ be a region and $P_{n+1}\in R$.
Then the line cycle of $P_{n+1}$ for the point arrangement $\mcl{P}_n\cup\{P_{n+1}\}$ is a $2$-standard consecutive structure and the cycle is independent of any point $P_{n+1}\in R$ and 
depends only on the region $R$. Moreover different regions of the convex $n\operatorname{-}$gon are associated to different cycles.
\end{theorem}
\begin{proof}
Since the point $P_{n+1}\in R$ and does not lie on the diagonals, the finite set $\mcl{P}_n\cup\{P_{n+1}\}$ is a point arrangement. An anticlockwise traversal around the point $P_{n+1}$ gives a cycle $(1=a_1a_2\ldots a_n)$ which has the property that it is obtained by interlacing in some manner two sequences $1<2<\ldots<i$ and $i+1<i+2<\ldots<n$ for some $2<i<n$ giving rise to a $2$-standard consecutive cycle. It is clear that the line cycle depends only on the region $R$. If $R$ and $S$ are two different regions then there exist three subscripts $i<j<k$ such that $\Gd P_iP_jP_k$ contains $R$ and 
$\Gd P_iP_jP_k$ does not contain $S$. In the cycle associated to $R$ the subscripts $i,j,k$ appear as the sub-cycle $(ikj)$ and in the cycle associated to $S$ the subscripts $i,j,k$ appear as the sub-cycle $(ijk)$ and clearly they are different.
\end{proof}
\begin{defn}[$2$-Standard Consecutive Cycle of a Region]
	Let $\mcl{P}_n=\{P_1,\ldots,$ $P_n\}$ be a point arrangement in $\mbb{R}^2$ such that the points form a convex $n\operatorname{-}$gon in the anticlockwise manner $P_1\lra P_2\lra \ldots \lra P_n\lra P_1$ and which has generic diagonals. Let $R$ be a region. The line cycle associated to any point $P\in R$ is defined to be the $2$-standard consecutive cycle of the region $R$.
	Using Theorem~\ref{theorem:LineCycleofaRegion}, the cycle is well defined, unique and has the $2$-standard consecutive structure.
\end{defn}
\begin{defn}[Isomorphism of Point Arrangements]
	Let $\mcl{P}^j_n=\{P^j_1,P^j_2,\ldots,$ $P^j_n\},j=1,2$ be two point arrangements in the plane $\mbb{R}^2$. Then a bijection $\gd:\mcl{P}^1_n\lra \mcl{P}^2_n$ is an isomorphism if for any four points $A,B,C,D\in \mcl{P}^1_n$, $D$ is in the interior of the triangle formed by $A,B,C$ if and only if $\gd(D)$ is in the interior of the triangle formed by $\gd(A),\gd(B),\gd(C)$.
	
	We say the isomorphism $\gd$ is an orientation preserving isomorphism if for any three points 
	the $A,B,C \in \mcl{P}^1_n$ the orientation $A\lra B \lra C \lra A$ of the triangle $\Gd ABC$ and the orientation $\gd(A) \lra \gd(B) \lra \gd(C) \lra \gd(A)$ of the triangle  $\Gd \gd(A)\gd(B)\gd(C)$ agree. In~\cite{MR2163782}, this is defined as an equivalence relation for order types of point arrangements. In this case $\mcl{P}^1_n$ and $\mcl{P}^2_n$ is said to have the same order type.
	
	We say the isomorphism $\gd$ is an orientation reversing isomorphism if for any three points 
	the $A,B,C \in \mcl{P}^1_n$ the orientation $A\lra B \lra C \lra A$ of the triangle $\Gd ABC$ and the orientation $\gd(A) \lra \gd(B) \lra \gd(C) \lra \gd(A)$ of the triangle  $\Gd \gd(A)\gd(B)\gd(C)$ disagree. 
\end{defn}
\begin{figure}[h]
	\centering
	\includegraphics[width = 1.0\textwidth]{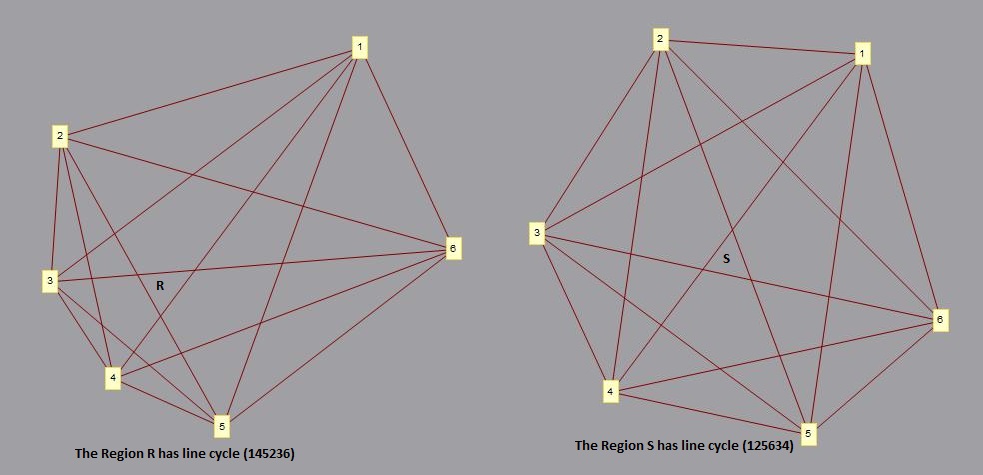}
	\caption{Two non-isomorphic regions $R$ and $S$ in hexagons}
	\label{fig:One}
\end{figure}
\begin{defn}[Isomorphism between two regions]
~\\
	For $j=1,2$, let $\mcl{P}^j_n=\{P_1^j,\ldots,P_n^j\}$ be two point arrangements in $\mbb{R}^2$ such that the points form a convex $n\operatorname{-}$gon in the anticlockwise manner $P^j_1\lra P^j_2\lra \ldots \lra P^j_n\lra P^j_1$ and both of which has generic diagonals. Let $R^j$ be a region in $\mcl{P}^j_n,j=1,2$. We say $R^1$ is isomorphic to $R^2$
	if the line cycle of the regions $R^1$ and $R^2$ are same, that is, for any two points 
	$P_{n+1}^j\in R^j,j=1,2$ the map $\gd:\mcl{P}^1_n\cup\{P^1_{n+1}\}\lra \mcl{P}^2_n\cup\{P^2_{n+1}\}$ given by $\gd(P_i^1)=P_i^2,1\leq i\leq n+1$ is an orientation preserving isomorphism of point arrangements.
\end{defn}

\begin{example}
In Figure~\emph{\ref{fig:One}}, $R$ and $S$ are two non-isomorphic regions of the two hexagons with line cycles $(145236),(125634)$ respectively.
	
\end{example}
Later in Theorem~\ref{theorem:PropertiesofCycles} we will observe that if two regions $R$ and $S$ of two convex $n\operatorname{-}$gons respectively are isomorphic then the regions $R$ and $S$ still need not have the same number of sides. The number of sides can vary.

Now we enumerate the $2$-standard consecutive cycles.
\begin{lemma}
	\label{lemma:NumberofTwoStandardConsecutivecycles}
	Let $T_n\subs S_n$ be the set of $2\operatorname{-}$standard consecutive $n\operatorname{-}$cycles in $S_n$.
	\begin{enumerate}
		\item We have \equ{\#(T_n)=2^{n-1}-n.}
		\item The number of non-isomorphic regions $R$ in a convex $n\operatorname{-}$gon is also $2^{n-1}-n$.   
	\end{enumerate}
\end{lemma}
\begin{proof}
	This proof of $(1)$ follows by counting the cardinality of $T_n$.	
	If the $2\operatorname{-}$standard consecutive structure is given by 
	\begin{itemize}
		\item $1<2<3<\ldots<j$,
		\item $j+1<j+2<\ldots<n$,
	\end{itemize}
	then the number of such cycles is given by $\binom{n-1}{j-1}-1$. 
	Hence the total number is given by \equ{\us{i=2}{\os{n-1}{\sum}}\bigg(\binom{n-1}{i-1}-1\bigg)=\us{i=0}{\os{n-1}{\sum}}\bigg(\binom{n-1}{i}-1\bigg)=2^{n-1}-n.}
\begin{figure}[h]
	\centering
	\includegraphics[width = 1.0\textwidth]{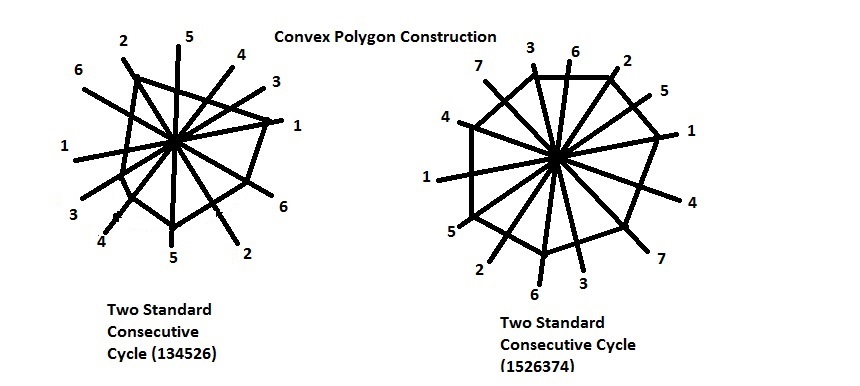}
	\caption{Two examples for $n=6,7$ and cycles $(134526),(1526374)$ respectively}
	\label{fig:Two}
\end{figure}	
	Now we construct for every given $2$-standard consecutive cycle, a convex $n\operatorname{-}$gon and a region $R$ inside it which has the given $2$-standard consecutive cycle. Let $(1=a_1a_2\ldots a_n)$ be a given $2$-standard consecutive cycle. The construction is done as follows. Choose a point $P$ in the plane and draw $n$ lines $L_{a_1},L_{a_2},\ldots,L_{a_n}$ passing through $P$ with increasing angles for $L_{a_i}$ with respect to $L_1=L_{a_1}$ by assuming $L_1$ is the $X\operatorname{-}$axis. Let $a_{i_j}=j,1\leq j\leq n$.
	Choose a point $P_1$ on the positive $X\operatorname{-}$axis. 
	Now traverse anticlockwise around the point $P$ cutting the lines $L_{a_2},L_{a_3}\ldots,L_{a_{(i_2-1)}}$ and choose a point $P_2$ on the ray that we have reached on $L_{a_{i_2}}=L_2$ and continue this process till we choose a point $P_n$ on a suitable ray of $L_{a_{i_n}}=L_n$. In this process we make sure that the we obtain a convex $n\operatorname{-}$gon $P_1\lra P_2\lra \ldots \lra P_n\lra P_1$ in the anticlockwise manner with generic diagonals. We refer to Figure~\ref{fig:Two} for two illustrative examples. This proves the lemma.	
\end{proof}

\section{\bf{Definite and Indefinite Cycles and their Characterization}}
\label{sec:DefIndefRegions}
We have seen in Theorem~\ref{theorem:NumberofRegions} that there are polynomial (in $n$) number of regions in any convex $n\operatorname{-}$gon with generic diagonals. However in Lemma~\ref{lemma:NumberofTwoStandardConsecutivecycles} we have seen that there are exponential (in $n$) number of $2\operatorname{-}$standard consecutive cycles. So we conclude that not every region labeled by $2\operatorname{-}$standard consecutive cycle occurs in every convex $n\operatorname{-}$gon with generic diagonals. This motivates the following definition.

\begin{defn}[Definite and Indefinite $2$-Standard Consecutive $n\operatorname{-}$cycle]
We say a $2$-standard consecutive $n\operatorname{-}$cycle $(1=a_1a_2\ldots a_n)$ is definite
if there is a corresponding region $R$ which occurs in every convex $n\operatorname{-}$gon with generic diagonals. Otherwise, we say the cycle is indefinite.	
\end{defn}
\begin{remark}
In the above definition, we have defined, definite cycles for convex $n$-gon with generic diagonals. If $\mcl{C}_n$ is the class of all convex $n$-gons in the plane and $(\mcl{C}_n)_{gen}$ is the class of all convex $n$-gons with generic diagonals then we will see in Theorem~\ref{theorem:Reducibility} that the set of definite cycles remains the same for both the classes of convex $n$-gons. Those which are not definte are the indefinite cycles. The combinatorial descriptions of both types of cycles are given in Theorem~\ref{theorem:DefIndefRegions}.
\end{remark}

\begin{example}
All regions labelled by $2$-standard consecutive cycles of a triangle, quadrilateral and a pentagon are definite. The first occurence of indefinite regions is when $n=6$. We refer to Figure~\emph{\ref{fig:One}}. The indefinite $2$-standard consecutive cycles are given by 
\equ{(145236),(125634).}
\end{example}

\subsection{\bf{Some Properties of $2$-Standard Consecutive Cycles Associated to Regions}}
We state the following theorem.
\begin{theorem}
	\label{theorem:PropertiesofCycles}
	Let $\mcl{P}_n=\{P_1,\ldots,P_n\}$ be a point arrangement in $\mbb{R}^2$ such that the points form a convex $n\operatorname{-}$gon in the anticlockwise manner $P_1\lra P_2\lra \ldots \lra P_n\lra P_1$ and which has generic diagonals. Let $R$ be a region with associated $2$-standard consecutive cycle $(1=a_1a_2\ldots a_n)$. Let the $2$-standard consective structure be given by 
	\equa{&1<2<\ldots<l,\\&(l+1)<(l+2)<\ldots<n.}
	Then 
	\begin{enumerate}[label=\emph{\arabic*}.]
		\item If $P_iP_j$ is a side of the region $R$ then the elements $i,j$ are consecutive in the cycle $(1=a_1a_2\ldots a_n)$.
		\item The $2$-standard consecutive cycle for the region $S$ on the other side of the region $R$ across $P_iP_j$ is obtained by swapping $i,j$ in the cycle $(1=a_1a_2\ldots a_n)$.
		\item Let $P_iP_j$ be the side of the region $R$ and $j\not\equiv i\pm 1\mod n$. Then either 
		$i\in \{1,2,\ldots,l\},j\in \{l+1,l+2,\ldots,n\}$ or $j\in \{1,2,\ldots,l\},i\in \{l+1,l+2,\ldots,n\}$.
		\item The converse is not true, that is, if $i,j$ occur consecutively in 
		$(1=a_1a_2\ldots a_n)$ such that $j\not\equiv i\pm 1\mod n, 1\leq i\leq l,l+1\leq j\leq n$ then $P_iP_j$ need not be a side of the region $R$.
		\item The number of sides of a region $R$ with the same cycle $(1=a_1a_2\ldots a_n)$ may vary from one convex $n\operatorname{-}$gon to another convex $n\operatorname{-}$gon provided it occurs in them.
		\item The region $R$ is contained in the triangle $\Gd P_iP_jP_k$ with $i<j<k$ if and only if the cycle $(1=a_1a_2\ldots a_n)$ contains $(ikj)$ as a sub-cycle.
	\end{enumerate}
\end{theorem}
\begin{proof}
~
\begin{enumerate}
\item The line cycle does not change for any point in $R$. Hence choosing a point close to $P_iP_j$ in the region $R$ we conclude that $i,j$ occurs consecutively in the line cycle. Moreover if we give anticlockwise orientation to the sides of $R$ and the directed side is $\ora{ij}$ then $i, j$ appear next to each other with $j$ first and $i$ second in the the cycle $(1=a_1a_2\ldots a_n)$.

\item If we go across the side $P_iP_j$ to a new region $S$ from $R$ then it is clear that there is swap of $i,j$ in the line cycle $(1=a_1a_2\ldots a_n)$ to obtain the line cycle for $S$.

\item If $P_iP_j$ is a side of the region $R$ with $j\not\equiv i\pm 1\mod n$ then it is not the side of the convex $n\operatorname{-}$gon. So there is a region $S$ adjacent to $R$ across $P_iP_j$. The cycle of $S$ is obtained by swapping $i,j$ in $(1=a_1a_2\ldots a_n)$. Now this is $2$-standard consecutive if and only if $i\in \{1,2,\ldots,l\},j\in \{l+1,l+2,\ldots,n\}$ or vice-versa.
\begin{figure}[h]
	\centering
	\includegraphics[width = 0.6\textwidth]{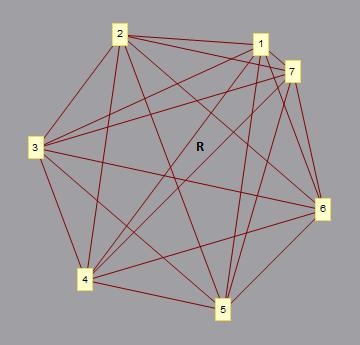}
	\caption{A heptagon and a region $R$ with cycle $(1526374)$ respectively}
	\label{fig:Three}
\end{figure}
\item We consider the example in Figure~\ref{fig:Three}. In this example the region $R$ is a pentagon with cycle $(1526374)$ and the $2$-standard consecutive structure $1<2<3<4; 5<6<7$ with $l=4$. The numbers $1,5$ appear consecutively in the cycle of $R$, however, $P_1P_5$ is not the side of the region $R$.

\item Consider the central region of a regular heptagon. It also has cycle $(1526374)$.
This central region is heptagon where as the region $R$ in Figure~\ref{fig:Three} is a pentagon.

\item Let $Q_1,Q_2,Q_3\in \mcl{P}_n$. Let $Q_4$ be in the interior of the convex $n\operatorname{-}$gon. Then $Q_4$ is in the interior of the triangle $\Gd Q_1Q_2Q_3$ oriented anticlockwise if and only if the line cycle of $Q_4$ for the point arrangement $\{Q_1,Q_2,Q_3,Q_4\}$ is $(132)$.   

\end{enumerate}
This completes the proof of the theorem.
\end{proof}
\subsection{\bf{Characterization of Definite and Indefinite Cycles}}
In this section we characterize the definite and indefinite $2$-standard consecutive $n\operatorname{-}$cycles combinatorially. We begin with a definition.
\begin{defn}[The Diagonal Distance of a $2$-Standard Consecutive Cycle]
Let $(1=a_1a_2\ldots a_n)$ be a $2$-standard consecutive $n\operatorname{-}$cycle. Let the $2$-standard consecutive structure be given by \equa{&1<2<\ldots<l,\\ &l+1<l+2<\ldots<n.}	
Consider the set $S$ of all pairs $\{i,j\}$ with $1\leq i\neq j\leq n$ such that 
\begin{enumerate}
\item $j\not\equiv i\pm 1\mod n$,
\item $i,j$ are consecutive in $(1=a_1a_2\ldots a_n)$ and 	
\item $i\in \{1,\ldots,l\},j\in \{l+1,\ldots,n\}$ or $j\in \{1,\ldots,l\},i\in \{l+1,\ldots,n\}$
\end{enumerate}
The diagonal distance of $(1=a_1a_2\ldots a_n)$ is defined as 
\equ{\us{\{i,j\}\in S}{\min} \{(i-j)\;\mathrm{mod}\;n,(j-i)\;\mathrm{mod}\;n\}}
Here the residue classes $\mathrm{mod}\;n$ are $\{0,1,\ldots,(n-1)\}$.
\end{defn}
\begin{example}
Now we consider an illustration of Theorem~\emph{\ref{theorem:DiagDistTwo}} where the cycles with diagonal distance two are mentioned by taking $n=6$. There are $2^{6-1}-6=26$ $2$-standard consecutive cycles. There are $24$ cycles with diagonal distance two. Now we list the cycles obtained in Case $1$ of Theorem~\emph{\ref{theorem:DiagDistTwo}}. 
	\begin{itemize}
		\item Move $1 \lra (134562),(145623),(156234),(162345)$.
		\item Move $2 \lra (132456),(134256),(134526),(134562)$.
		\item Move $3 \lra (124356),(124536),(124563),(132456)$.
		\item Move $4 \lra (123546),(123564),(142356),(124356)$.
		\item Move $5 \lra (123465),(152346),(125346),(123546)$.
		\item Move $6 \lra (162345),(126345),(123645),(123465)$. 
	\end{itemize}
	These give $n^2-3n=6^2-3.6=18$ cycles which are definite. These are the outermost layer regions of the convex $n\operatorname{-}$gon where $n=6$.
	Now we list the cycles obtained in Case $2$ of Theorem~\emph{\ref{theorem:DiagDistTwo}}.
	\begin{itemize}
		\item Move $1$ and swap $2,6 \lra (134526),(145263),(152634),(126345)$.
		\item Move $2$ and swap $1,3 \lra (124563),(142563),(145263),(145623)$.
		\item Move $3$ and swap $2,4 \lra (142356),(142536),(142563),(134256)$.
		\item Move $4$ and swap $3,5 \lra (125346),(125364),(142536),(124536)$.
		\item Move $5$ and swap $4,6 \lra (123645),(152364),(125364),(123564)$.
		\item Move $6$ and swap $1,5 \lra (156234),(152634),(152364),(152346)$. 
	\end{itemize}
	These give in addition $n^2-5n=6^2-5.6=6$ cycles which are definite. These six cycles correspond to the second outermost layer regions of the convex $n\operatorname{-}$gon where $n=6$.
	We also have two indefinite cycles given by $(145236),(125634)$.
	These total to $18+6+2=26$ cycles.
\end{example}
Next we characterize the $2$-standard consecutive cycle which has diagonal distance two.
\begin{theorem}
\label{theorem:DiagDistTwo}
Let $(1=a_1a_2\ldots a_n)$ be a $2$-standard consecutive $n\operatorname{-}$cycle with the $2$-standard consecutive structure given by $1<\ldots<l;\;\; l+1<\ldots<n$. Then the diagonal distance of the cycle is two if and only if it is obtained from the cycle $(12\ldots n)$
in the following two ways.
\begin{enumerate}[label=\emph{\arabic*}.]
\item Considering the cyclic notation of the cycle $(12\ldots n)$ on a circle in an anticlockwise manner, and moving $i$ forward \emph{(}anticlockwise\emph{)} for some $1\leq i\leq n-1$ in a finite number of steps to any position after $(i+1)$ and before $(i-1)$. 
\item From any cycle obtained in the previous step by moving $i$, we swap the adjacent elements $(i-1),(i+1)$.
\end{enumerate}
\end{theorem}
\begin{proof}
We prove the reverse implication $(\La)$ first. Any cycle obtained in steps $(1),(2)$ have diagonal distance two since $(i-1),(i+1)$ considered cyclically are consecutive and either 
$(i-1)\in \{1,\ldots,l\},i+1\in \{l+1,l+2,\ldots,n\}$ or $(i+1)\in \{1,\ldots,l\},i-1\in \{l+1,l+2,\ldots,n\}$. This proves that the diagonal distance is two.

Now we prove the forward implication. Suppose the diagonal distance is two. Then first we construct a convex $n\operatorname{-}$gon with vertices $P_1\lra P_2\lra \ldots \lra P_n\lra P_1$ in this anticlockwise manner such that the cycle $(1=a_1a_2\ldots a_n)$ appears as a region using Figure~\ref{fig:Two}. 
Since the diagonal distance is two, there exists $1\leq i\leq n$ such that $i-1,i+1$ appear consecutively and either $(i-1)\in \{1,\ldots,l\},i+1\in \{l+1,l+2,\ldots,n\}$ or $(i+1)\in \{1,\ldots,l\},i-1\in \{l+1,l+2,\ldots,n\}$. Now using the convex $n\operatorname{-}$gon, by applying moves of the type mentioned in step $(1),(2)$ on the cycle, we just cross the regions to reach the outside of convex $n\operatorname{-}$gon via crossing either the side $P_{i-1}P_i$ or $P_iP_{i+1}$. Here the cycle we arrive at must be the cycle $(12\ldots n)$. Hence this proves the forward implication thereby completing the proof of the theorem.
\end{proof}
Now we prove a lemma which is used in the proof of Theorem~\ref{theorem:DefIndefRegions}. This lemma describes certain $2$-standard consecutive $(n-1)\operatorname{-}$cycles of diagonal distance two which upon adding $n$ gives two standard consecutive $n\operatorname{-}$cycles of diagonal distance two.
\begin{lemma}
\label{lemma:TwoStandardExtension}
Let $7\leq n\in \mbb{N}$. Let $(1=a_1a_2\ldots a_{n-1})$ be a $2$-standard consecutive $(n-1)\operatorname{-}$cycle with diagonal distance two. For some $1\leq i\leq n-1$ let \equ{(i-1)\;\mathrm{mod}\;(n-1),(i+1)\;\mathrm{mod}\;(n-1)\in \{1,\ldots,(n-1)\}} appear consecutively in the cycle and suppose $i\nin \{1,2,(n-2),(n-1)\}$, that is, \equa{\{(i-1)\;\mathrm{mod}\;(n-1),&(i+1)\;\mathrm{mod}\;(n-1)\}\\ \nin\{\{1,n-2\},&\{2,n-1\},\{1,3\},\{n-3,n-1\}\}.} If we add $n$ to the cycle and obtain a $2$-standard consecutive $n\operatorname{-}$cycle then it also has diagonal distance two.
\end{lemma}
\begin{proof}
If we add $n$ to $(1=a_1a_2\ldots a_{n-1})$ in the right of both $(i-1)\;\mathrm{mod}\;(n-1),(i+1)\;\mathrm{mod}\;(n-1)$, the cycle still has diagonal distance two unless $\{(i-1)\;\mathrm{mod}\;(n-1),(i+1)\;\mathrm{mod}\;(n-1)\} \in\{\{2,n-1\},\{1,n-2\}\}$. We cannot add $n$ after $1$ and to the left of both of them as the resulting cycle will not be $2$-standard consecutive. Now if we add $n$ in between $(i-1)\;\mathrm{mod}\;(n-1)$ and $(i+1)\;\mathrm{mod}\;(n-1)$ then only the following possibilities occur.
\begin{enumerate}[label=(\arabic*)]
\item $(1=a_1a_2\ldots a_{n-1}n)=(1=a_1\ldots 3n)$ with $a_{n-1}=3$ or
\item $(1=a_1a_2\ldots a_{n-1}n)=(1(n-1)2\ldots (n-2)n)$ with $a_2=(n-1),a_j=j-1,3\leq j\leq n-1$  or
\item $(1=a_1a_2\ldots a_{n-1}n)=(12\ldots j(n-1)(j+1)\ldots (n-2)n)$ for some $2\leq j\leq (n-3)$
\item $(1=a_1\ldots (n-1)n2 \ldots a_{n-1})=(1j\ldots (n-2)(n-1)n2\ldots (j-1))$ with $4\leq j\leq n-2$ or
\item $(1=a_1a_2na_3\ldots a_{n-1})=(1(n-1)n2\ldots (n-2))$ with $a_2=(n-1),a_j=j-1,3\leq j\leq n-1$ or
\item $(1=a_1a_2\ldots a_{n-2}na_{n-1})=(13\ldots (n-1)n2)$ with $a_{n-1}=2,a_j=(j+1),2\leq j\leq n-2$ or
\item $(1=a_1a_2\ldots a_{n-4}a_{n-3}na_{n-2}a_{n-1})=(123\ldots (n-4)(n-1)n(n-3)(n-2))$ with $a_j=j,2\leq j\leq n-4,a_{n-3}=n-1,a_{n-2}=n-3,a_{n-1}=n-2$.
\item $(1=a_1a_2\ldots a_{n-3}a_{n-2}na_{n-1})=(1\ldots j(n-2)(j+1)\ldots(n-4)(n-1)n(n-3))$ for some $2\leq j\leq n-5,a_{n-3}=n-4,a_{n-2}=n-1,a_{n-1}=n-3$.
\end{enumerate}
In these cases we have $\{(i-1)\;\mathrm{mod}\;(n-1),(i+1)\;\mathrm{mod}\;(n-1)\} \in\{\{1,n-2\},\{2,n-1\},\{1,3\},\{n-1,n-3\}\}$. This proves the lemma.
\end{proof}
Now we state and prove the main theorem of the article. Later in Theorem~\ref{theorem:Reducibility}, we will formulate the main theorem in the language of matroids.
\begin{thmOmega}
\namedlabel{theorem:DefIndefRegions}{$\Gom$}
Let $(1=a_1a_2\ldots a_n)$ be a $2$-standard consecutive $n\operatorname{-}$cycle. 
\begin{enumerate}[label=(\Alph*)]
\item The cycle $(1=a_1a_2\ldots a_n)$ is definite if and only if it is obtained from the cycle $(12\ldots n)$ in the following three ways.
\begin{enumerate}
\item Considering the cyclic notation of the cycle $(12\ldots n)$ on a circle in an anticlockwise manner, and moving $i$ forward (anticlockwise) for some $1\leq i\leq n-1$ in a finite number of steps to any position after $(i+1)$ and before $(i-1)$. The diagonal distance of these cycles is two.
\item From any cycle obtained in the previous step by moving $i$, we swap the adjacent elements $(i-1),(i+1)$. The diagonal distance of these cycles is also two.
\item $n=7$ and $(1=a_1a_2\ldots a_n)=(1526374)$. This is the only cycle which has diagonal distance more than two and is definite. This phenomenon occurs only when $n=7$.
\end{enumerate}
\item The cycle $(1=a_1a_2\ldots a_n)$ is indefinite if and only if there exists \equ{1\leq i_1<i_2<i_3<i_4<i_5<i_6\leq n} with the following property that if for some $j\in \{0,1,2,3,4,5\}$ \equ{a_{i_{j+1}}=\min\{a_{i_1},a_{i_2},a_{i_3},a_{i_4},a_{i_5},a_{i_6}\}\text{ with } a_{i_{j+t}}=a_{i_{((j+t)\;\mathrm{mod}\; 6)}}} where 
$((j+t)\;\mathrm{mod}\;6) \in \{1,2,3,4,5,6\}$ then we have either \equ{a_{i_{j+1}}<a_{i_{j+4}}<a_{i_{j+5}}<a_{i_{j+2}}<a_{i_{j+3}}<a_{i_{j+6}}} or \equ{a_{i_{j+1}}<a_{i_{j+2}}<a_{i_{j+5}}<a_{i_{j+6}}<a_{i_{j+3}}<a_{i_{j+4}}.}
Here in the subscripts the local cycles $(145236),(125634)$ appear which are the indefinite cycles for $n=6$.
\end{enumerate}
\end{thmOmega}
The proof of Theorem~\ref{theorem:DefIndefRegions} is given after the following example.
\begin{example}
We illustrate Theorem~\ref{theorem:DefIndefRegions}$\emph{(B)}$ in this example.
Clearly the cycles $(145236),(125634)$ are indefinite using Theorem~\ref{theorem:DefIndefRegions}$\emph{(B)}$.
Consider the $2$-standard consecutive cycle \equ{(1=a_1a_2a_3a_4a_5a_6a_7a_8)=(15263748).} Choose $i_1=1,i_2=2,i_3=4,i_4=5,i_5=7,i_6=8$. Choose $j=0$. We have 
\equa{&a_{i_{j+1}}=a_{i_1}=a_1=1<a_{i_{j+4}}=a_{i_4}=a_5=3<a_{i_{j+5}}=a_{i_5}=a_7=4<\\
&a_{i_{j+2}}=a_{i_2}=a_2=5<a_{i_{j+3}}=a_{i_3}=a_4=6<a_{i_{j+6}}=a_{i_6}=a_8=8.}
So using  Theorem~\ref{theorem:DefIndefRegions}$\emph{(B)}$ we have that the cycle is indefinite. Also observe that this can be expressed by the fact that $(815634)$ is a sub-cycle.

Consider the $2$-standard consecutive cycle \equ{(1=a_1a_2a_3a_4a_5a_6a_7a_8)=(15263784).} Choose
$i_1=2,i_2=3,i_3=5,i_4=6,i_5=7,i_6=8$. Choose $j=1$. We have 
\equa{&a_{i_{j+1}}=a_{i_2}=a_3=2<a_{i_{j+2}}=a_{i_3}=a_5=3<a_{i_{j+5}}=a_{i_6}=a_8=4<\\
&a_{i_{j+6}}=a_{i_1}=a_2=5<a_{i_{j+3}}=a_{i_4}=a_6=7<a_{i_{j+4}}=a_{i_5}=a_7=8.}
So using  Theorem~\ref{theorem:DefIndefRegions}$\emph{(B)}$ we have that the cycle is indefinite. Also observe that this can be expressed by the fact that $(452378)$ is a sub-cycle.
\end{example}
\begin{remark}
From Theorem~\ref{theorem:DefIndefRegions} and the property given in Theorem~\ref{theorem:DefIndefRegions}(B), it is clear that a region $R$ has an indefinite cycle if and only if it does not occur in some convex $n$-gon with generic diagonals. Also this fact becomes more clear in the proof of Theorem~\ref{theorem:DefIndefRegions}. 
\end{remark}
\begin{proof}[Proof of Theorem~\ref{theorem:DefIndefRegions}]
First we prove the reverse $(\La)$ implication in (A). 

The cycles obtained in A(a)
and A(b) exactly correspond to the regions for which one of the following $n$ diagonals \equ{P_1P_3,P_2P_4,\ldots,P_{n-1}P_1,P_nP_2} is a side and these clearly are definite regions.
These cycles have diagonal distance exactly two. There are $2n^2-8n$ such cycles with corresponding regions. For $n=7$ there are $2^{7-1}-7=57$ $2$-standard consecutive cycles. There are $\frac{(7-1)(7-2)(7^2-3*7+12)}{24}=50$ regions in a heptagon with generic diagonals. Now in Figure~\ref{fig:One} we have seen for the hexagon that the cycles $(145236)$ and $(125634)$ are indefinite and mutually exclusive, that is, if one occurs then the other does not occur. Extending this scenario for $n=7$ we conclude that there are seven pair of mutually exclusive $2$-standard consecutive cycles. They are obtained by cyclically shifting as follows.
\begin{enumerate}[label=(\Roman*)]
\item First pair $(1523674),(1256347)$. We ignore $1$ and the oriented triangle $\Gd \ora{P_5P_2}\ora{P_3P_6}\ora{P_7P_4}$ containing the region $(1523674)$ is oriented clockwise 
where as the oriented triangle $\Gd \ora{P_5P_2}\ora{P_3P_6}\ora{P_7P_4}$ containing the region $(1256347)$ is anticlockwise.
We illustrate this in Figure~\ref{fig:Four}.
\begin{figure}[h]
	\centering
	\includegraphics[width = 0.8\textwidth]{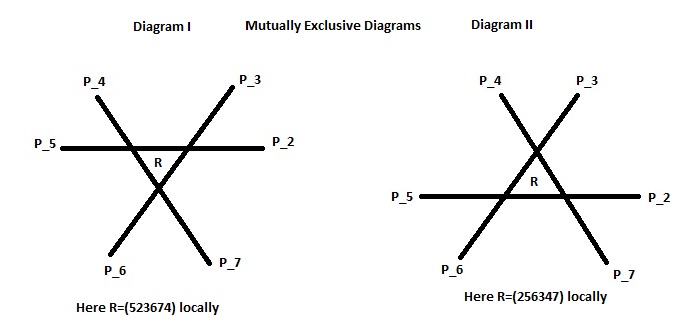}
	\caption{Local Triangles ignoring $1$ containing mutually exclusive cycles/regions $(1523674),(1256347)$ respectively}
	\label{fig:Four}
\end{figure}
\item Second pair $(1526347),(1236745)$. We ignore $2$ and the oriented triangle $\Gd \ora{P_1P_5}\ora{P_6P_3}\ora{P_4P_7}$ containing the region $(1526347)$ is oriented clockwise 
where as the oriented triangle $\Gd \ora{P_1P_5}\ora{P_6P_3}\ora{P_4P_7}$ containing the region $(1236745)$ is anticlockwise.
\item Third pair $(1263745),(1562347)$. We ignore $3$ and the oriented triangle $\Gd \ora{P_5P_1}\ora{P_2P_6}\ora{P_7P_4}$ containing the region $(1263745)$ is oriented clockwise 
where as the oriented triangle $\Gd \ora{P_5P_1}\ora{P_2P_6}\ora{P_7P_4}$ containing the region $(1562347)$ is anticlockwise.
\item Fourth pair $(1562374),(1267345)$. We ignore $4$ and the oriented triangle $\Gd \ora{P_1P_5}\ora{P_6P_2}\ora{P_3P_7}$ containing the region $(1562374)$ is oriented clockwise 
where as the oriented triangle $\Gd \ora{P_1P_5}\ora{P_6P_2}\ora{P_3P_7}$ containing the region $(1267345)$ is anticlockwise. 
\item Fifth pair $(1526734),(1456237)$. We ignore $5$ and the oriented triangle $\Gd \ora{P_4P_1}\ora{P_2P_6}\ora{P_7P_3}$ containing the region $(1526734)$ is oriented clockwise 
where as the oriented triangle $\Gd \ora{P_4P_1}\ora{P_2P_6}\ora{P_7P_3}$ containing the region $(1456237)$ is anticlockwise. 
\item Sixth pair $(1452637),(1256734)$. We ignore $6$ and the oriented triangle $\Gd \ora{P_1P_4}\ora{P_5P_2}\ora{P_3P_7}$ containing the region $(1452637)$ is oriented clockwise 
where as the oriented triangle $\Gd \ora{P_1P_4}\ora{P_5P_2}\ora{P_3P_7}$ containing the region $(1256734)$ is anticlockwise. 
\item Seventh pair $(1256374),(1452367)$. We ignore $7$ and the oriented triangle $\Gd \ora{P_4P_1}\ora{P_2P_5}\ora{P_6P_3}$ containing the region $(1256374)$ is oriented clockwise 
where as the oriented triangle $\Gd \ora{P_4P_1}\ora{P_2P_5}\ora{P_6P_3}$ containing the region $(1452367)$ is anticlockwise. 
\end{enumerate}
Now there are $2n^2-8n=2*7^2-8*7=42$ definite cycles whose regions definitely occur using A(a),A(b). These cycles can be written down. Out of the remaining $15$ cycles there are seven mutually exclusive pairs. Totally there are $50$ regions that occur in a heptagon with generic diagonals. Hence there is one more definite cycle which occurs that is cycle $(1526374)$.
This completes the proof of reverse implication $(\La)$ of Theorem~\ref{theorem:DefIndefRegions}(A).

Now we prove the reverse implication $(\La)$ in (B). In both the cases mentioned in (B) we can obtain a mutually exclusive cycle for the cycle $(1=a_1a_2\ldots a_n)$ by orienting the local triangle the other way similar to the diagrams given in Figure~\ref{fig:Four}. We illustrate this in Figure~\ref{fig:Five}.
\begin{figure}[h]
	\centering
	\includegraphics[width = 1.0\textwidth]{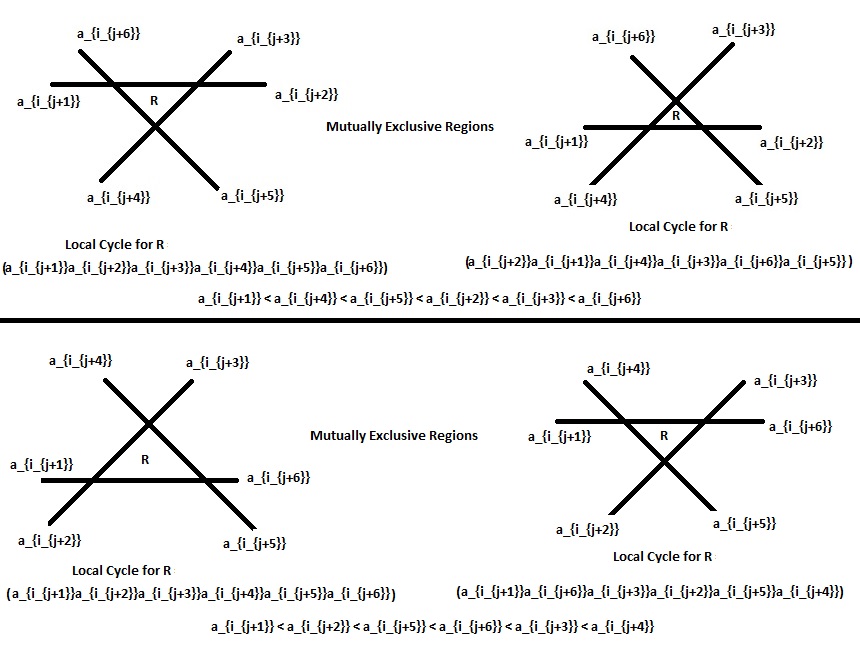}
	\caption{Mutually Exclusive Indefinite Regions}
	\label{fig:Five}
\end{figure}
This proves that the cycle $(1=a_1a_2\ldots a_n)$ is indefinite. This completes the proof of reverse implication $(\La)$ of Theorem~\ref{theorem:DefIndefRegions}(B).

To complete the proof of Theorem~\ref{theorem:DefIndefRegions} it is enough to prove that if for a $2$-standard consecutive cycle $(1=a_1a_2\ldots a_n)$, the diagonal distance is more than two and $n\geq 8$ (Note: $n\neq 7$) then the cycle is indefinite and satisfies the property mentioned in Theorem~\ref{theorem:DefIndefRegions}(B). 

First we show that for $n=8$ the only definite cycles are those with diagonal distance two explicitly.
For this we list all the $120$ $2$-standard consecutive cycles as a union of fifteen orbits each containing eight cycles under the  cyclic shift action of $\Z/8\Z$.
The cycles in the first eight orbits $1,\ldots,8$ have diagonal distance two and satisfy the conditions of Theorem~\ref{theorem:DefIndefRegions} in A(a),A(b).  The cycles in the remaining seven orbits $9,\ldots,15$ satisfy the conditions of Theorem~\ref{theorem:DefIndefRegions}(B). The cycles in these orbits do not have diagonal distance two. We consider the one standard consecutive cycle $(12345678)$ on a circle in anticlockwise cyclic manner.
\begin{enumerate}
\item Move element $i$ forward (anticlockwise) by one position in $(12345678)$ for $1\leq i\leq 8$.
This is also same as moving another element $i$ forward by six positions in $(12345678)$  for $1\leq i\leq 8$.\\ $(13456782),(13245678),(12435678),(12354678),\\(12346578),(12345768),(12345687),(18234567)$.

\item Move element $i$ forward by two positions in $(12345678)$ for $1\leq i\leq 8$. This is also same as moving another element $i$ forward by one position in $(12345678)$ and swapping $(i-1),(i+1)$ considered cyclically for $1\leq i\leq 8$. They are given as: \\ $(14567823),(13425678),(12453678),(12356478),\\(12346758),(12345786),(17234568),(12834567)$.

\item Move element $i$ forward by three positions in $(12345678)$ for $1\leq i\leq 8$.\\ $(15678234),(13452678),(12456378),(12356748),\\ (12346785),(16234578),(12734568),(12384567)$.

\item Move element $i$ forward by four positions in $(12345678)$ for $1\leq i\leq 8$.\\ $(16782345),(13456278),(12456738),(12356784),\\ (15234678),(12634578),(12374568),(12348567)$.

\item Move element $i$ forward by five positions in $(12345678)$ for $1\leq i\leq 8$.\\ $(17823456),(13456728),(12456783),(14235678),\\  (12534678),(12364578),(12347568),(12345867)$.

\item Move element $i$ forward by two positions in $(12345678)$ and swap $(i-1),(i+1)$ considered cyclically for $1\leq i\leq 8$. This is also same as moving another element $i$ forward by five positions in $(12345678)$ and swapping $(i-1),(i+1)$ considered cyclically for $1\leq i\leq 8$. They are given as: \\ $(14567283),(14256783),(14253678),(12536478),\\ (12364758),(12347586), (17234586),(17283456)$.

\item Move element $i$ forward by three positions in $(12345678)$ and swap $(i-1),(i+1)$ considered cyclically for $1\leq i\leq 8$.\\ $(15672834),(14526783),(14256378),(12536748),\\ (12364785),(16234758),(12734586),(17238456)$.

\item Move element $i$ forward by four positions in $(12345678)$ and swap $(i-1),(i+1)$ considered cyclically for $1\leq i\leq 8$.\\ $(16728345),(14562783),(14256738),(12536784),\\(15236478),(12634758),(12374586),(17234856)$.

\item The following cycles satisfy the condition of Theorem~\ref{theorem:DefIndefRegions}(B). For example the cycle $(12563478)$ contains the sub-cycle $(125634)$.\\
$(12563478),(12367458),(12347856),(16723458),\\ (12783456),(14567238),(12567834),(14523678)$.

\item The following cycles satisfy the condition of Theorem~\ref{theorem:DefIndefRegions}(B). For example the cycle $(12563748)$ contains the sub-cycle $(125634)$.\\
$(12563748),(12367485),(16234785),(16273458),\\ (12738456),(15672384),(15267834),(14526378)$.

\item The following cycles satisfy the condition of Theorem~\ref{theorem:DefIndefRegions}(B). For example the cycle $(12567348)$ contains the sub-cycle $(125634)$.\\
$(12567348),(12367845),(15623478),(12673458),\\ (12378456),(15672348),(12678345),(14562378)$.

\item The following cycles satisfy the condition of Theorem~\ref{theorem:DefIndefRegions}(B). For example the cycle $(12567384)$ contains the sub-cycle $(125634)$.\\ $(12567384),(15236784), (15263478),(12637458),\\ (12374856),(16723485),(16278345),(14562738)$.

\item The following cycles satisfy the condition of Theorem~\ref{theorem:DefIndefRegions}(B). For example the cycle $(12563784)$ contains the sub-cycle $(125634)$.\\ $(12563784),(15236748),(12634785),(16237458),\\ (12734856),(16723845),(15627834),(14526738)$.

\item The following cycles satisfy the condition of Theorem~\ref{theorem:DefIndefRegions}(B). For example the cycle $(12637485)$ contains the sub-cycle $(126745)$ locally ignoring $3,8$.\\ $(12637485),(16237485),(16273485),(16273845),\\ (15627384),(15267384),(15263784),(15263748)$.

\item The following cycles satisfy the condition of Theorem~\ref{theorem:DefIndefRegions}(B). For example the cycle $(12673845)$ contains the sub-cycle $(126745)$ locally ignoring $3,8$.\\ $(12673845),(15623784),(15267348),(12637845),\\ (15623748),(12673485),(16237845),(15627348)$.
\end{enumerate} 
This completes the proof of Theorem~\ref{theorem:DefIndefRegions} for $n=8$ the base case of the induction step.

Now we show by induction on $k=n\geq 9$ that if the cycle is indefinite then the diagonal distance is not two and satisfies the criterion given in Theorem~\ref{theorem:DefIndefRegions}(B).

Using Lemma~\ref{lemma:TwoStandardExtension}, we first consider cycles of diagonal distance two and hence definite cycles for $k=n-1\geq 8$ of the form $(1=a_1a_2\ldots a_{n-1})$ which satisfy one of the following properties.
\begin{enumerate}[label=(\alph*)]
\item The cycle contains $2$ and $(n-1)$ consecutively with $(n-1)$ first and $2$ next (has diagonal distance two). 
\item The cycle contains $2$ and $(n-1)$ consecutively with $2$ first and $(n-1)$ next (has diagonal distance two). 
\item The cycle contains $1$ and $(n-2)$ consecutively with $1$ first and $(n-2)$ next to it (which has diagonal distance two).
\item The cycle contains $1$ and $(n-2)$ consecutively with $1$ first and $(n-2)$ at the end (which has diagonal distance two).
\item The cycle contains $(n-1)$ and $(n-3)$ consecutively with $(n-1)$ first and $(n-3)$ next 
(has diagonal distance two).
\item The cycle contains $(n-3)$ and $(n-1)$ consecutively with $(n-3)$ first and $(n-1)$ next 
(has diagonal distance two).
\item The cycle contains $1$ and $3$ consecutively with $1$ first and $3$ next to it (which has diagonal distance two). 
\item The cycle contains $1$ and $3$ consecutively with $1$ first and $3$ at the end (which has diagonal distance two). 
\end{enumerate}
These cycles upon adding $n$ may or may not remain definite. 
The remaining definite cycles upon adding $n$ will remain definite using Lemma~\ref{lemma:TwoStandardExtension}. We will later consider indefinite $(n-1)\operatorname{-}$cycles.
We mention the various cases (i)-(xvi) and prove in each case that upon adding $n$ the resulting $2$-standard consecutive cycle will either have diagonal distance two and hence remain definite or it does not have diagonal distance two and satisfies the criterion of Theorem~\ref{theorem:DefIndefRegions} (so becomes indefinite).

\begin{enumerate}[label=(\roman*)]
\item If the cycle $(1=a_1a_2\ldots a_{n-1})$ contains $2,(n-1)$ consecutively with $(n-1)$ first and $2$ next (has diagonal distance two) then the cycle is of the following form  
\equ{(1j\ldots (n-2)(n-1)23 \ldots (j-1))\text{ for some }4\leq j \leq n-2 \text{ or }}
\equ{\text{either }(134\ldots (n-1)2)\text{ or } (1(n-1)23\ldots(n-2)).}
It is clear that upon adding $n$ to $(134\ldots (n-1)2)$ the resulting $2$-standard consecutive cycle has diagonal distance two and hence definite. If we add $n$ in between $(n-1)$ and $2$ or just next to $2$ in any of these the resulting $2$-standard consecutive cycle also has diagonal distance two and hence definite. If we add $n$ to $(1(n-1)23\ldots(n-2))$ anywhere after $3$ then the resulting $2$-standard consecutive cycle also has diagonal distance two and hence definite. If we add $n$ to $(1j\ldots (n-2)(n-1)23\ldots (j-1))$ with $4\leq j\leq n-2$, anywhere after $3$ then the resulting $2$-standard consecutive cycle does not have diagonal distance two and is indefinite because it has $(23n1(n-2)(n-1))$ as a sub-cycle which satisfies the criterion of Theorem~\ref{theorem:DefIndefRegions}(B).
\item $2,(n-1)$ appear in this order next to $1$ as follows.
Now consider the cycle
\equ{\big(12(n-1)34\ldots(n-4)(n-3)(n-2)\big).} If we add $n$ anywhere after $(n-1)$ and before $(n-3)$ the resulting $2$-standard cycle does not have diagonal distance two and is indefinite because it has $(12(n-1)n(n-3)(n-2))$ as a sub-cycle which satisfies the criterion of Theorem~\ref{theorem:DefIndefRegions}(B).
If we add $n$ in between $(n-3)$ and $(n-2)$ or after $(n-2)$ then the cycle is definite and has diagonal distance two.

\item $2,(n-1)$ appear in this order in the middle as follows.
Now consider for some $3\leq a \leq n-4$ the cycle 
\equ{(1(a+2)\ldots (n-2)2(n-1)3 \ldots a(a+1)).} If we add $n$ anywhere between $(n-1)$ and $a$, then the resulting $2$-standard cycle does not have diagonal distance two and is indefinite because it has $(12(n-1)na(a+1))$ as a sub-cycle which satisfies the criterion of Theorem~\ref{theorem:DefIndefRegions}(B). If we add $n$ in between $a$ and $a+1$ and obtain the $2$-standard consecutive cycle $(15\ldots (n-2)2(n-1)3n4)$ for $a=3$ then it does not have diagonal distance two and is indefinite because it has $(23n156)$ as a sub-cycle which satisfies the criterion of Theorem~\ref{theorem:DefIndefRegions}(B) 
since $n\geq 8$. This argument does not work for $n=7$. This is because $(152634)$ has diagonal distance two for $k=6$ and if we add $n=7$ in between $(n-1)=6$ and $4$ we get $(1526374)$ which does not have $(237156)$ as a sub-cycle. In fact, $(1526374)$ does not have diagonal distance two and also does not satisfy the criterion of Theorem~\ref{theorem:DefIndefRegions}(B). It is a definite cycle for $n=7$ and this is the only exception phenomenon.
If we add $n$ after $(a+1)$ or for $a>3$ we add $n$ in between $a$ and $(a+1)$ then the resulting $2$-standard cycle does not have diagonal distance two and is indefinite because it has $(34n1(n-2)(n-1))$ as a sub-cycle which satisfies the criterion of Theorem~\ref{theorem:DefIndefRegions}(B). 

\item $2,(n-1)$ appear in this order at the last but one or at the end positions as follows.
Now consider the cycles 
\equ{(14\ldots (n-2)2(n-1)3) \text{ and }(134\ldots (n-2)2(n-1)).}
These cycles are definite as $1,3$ appear consecutively and has diagonal distance two.

\item $1,(n-2)$ appear in the beginning and $(n-1)$ appears at the end as follows.
Now consider the cycle
\equ{\big(1(n-2)23\ldots(n-3)(n-1)\big).} $n$ can only be added at the end and the resulting $2$-standard cycle is definite and has diagonal distance two and it has $(n-3),(n-1)$ as consecutive.

\item $1,(n-2)$ appear in the beginning and $(n-1)$ appears at the last but one position as follows.
Now consider the cycle
\equ{\big(1(n-2)23\ldots(n-4)(n-1)(n-3)\big).} If $n$ is added at the end then the
the resulting $2$-standard cycle is definite and has diagonal distance two and it has $(n-3),(n-1)$ as consecutive. If $n$ is added in between $(n-1)$ and $(n-3)$ then the resulting $2$-standard cycle does not have diagonal distance two and is indefinite because it has $(23(n-1)n(n-3)(n-2))$ as a sub-cycle which satisfies the criterion of Theorem~\ref{theorem:DefIndefRegions}(B).

\item $1,(n-2)$ appear in the beginning and $(n-1)$ appears in between $3$ and $(n-4)$ as follows.
Now consider the cycle
\equ{\big(1(n-2)23\ldots(n-1)\ldots(n-4)(n-3)\big).}
If we add $n$ at the end then the resulting $2$-standard cycle does not have diagonal distance two and is indefinite because it has $((n-4)(n-3)n1(n-2)(n-1))$ as a sub-cycle which satisfies the criterion of Theorem~\ref{theorem:DefIndefRegions}(B). If we add $n$ after $(n-1)$ and before $(n-3)$ then the resulting $2$-standard cycle does not have diagonal distance two and is indefinite because it has $(23(n-1)n(n-3)(n-2))$ as a sub-cycle which satisfies the criterion of Theorem~\ref{theorem:DefIndefRegions}(B).

\item $1,(n-2)$ appear in the beginning and $(n-1)$ appears in between $2$ and $3$ as follows. 
Now consider the cycle
\equ{\big(1(n-2)2(n-1)34\ldots(n-4)(n-3)\big).}
If we add $n$ anywhere after $4$ then the resulting $2$-standard cycle does not have diagonal distance two and is indefinite because it has $(34n1(n-2)(n-1))$ as a sub-cycle which satisfies the criterion of Theorem~\ref{theorem:DefIndefRegions}(B). 
If we add $n$ in between $(n-1)$ and before $4$ then the resulting $2$-standard cycle does not have diagonal distance two and is indefinite because it has $(12(n-1)n45)$ (Note: $n\geq 8$) as a sub-cycle which satisfies the criterion of Theorem~\ref{theorem:DefIndefRegions}(B). For $n=7$, this argument does not work.
This is because $(152634)$ has diagonal distance two for $k=6$ and if we add $n=7$ in between $(n-1)=6$ and $4$ we get $(1526374)$ which does not have $(126745)$ as a sub-cycle. In fact $(1526374)$ does not have diagonal distance two and also does not satisfy the criterion of Theorem~\ref{theorem:DefIndefRegions}(B). It is a definite cycle for $n=7$ and this is the only exception phenomenon.

\item $1,(n-2)$ appear in the beginning and $(n-1)$ appears in between $(n-2)$ and $2$ as follows. 
Now consider the cycle
\equ{\big(1(n-2)(n-1)23\ldots(n-4)(n-3)\big).}
If we add $n$ anywhere after $3$ then the resulting $2$-standard cycle does not have diagonal distance two and is indefinite because it has $(23n1(n-2)(n-1))$ as a sub-cycle which satisfies the criterion of Theorem~\ref{theorem:DefIndefRegions}(B). 
If we add $n$ in between $2$ and $3$ or in between $(n-1)$ and $2$ then the
the resulting $2$-standard cycle has diagonal distance two and is definite with $2,n$ as consecutive.
\item $1,(n-2)$ appear consecutively with $1$ first and $(n-2)$ at the end as follows.
Now consider the cycle 
\equ{\big(1 \ldots (n-1) \ldots (n-2)\big).}
If we add $n$ at the end after $(n-2)$ or just before $(n-2)$ then the resulting two standard cycle has diagonal distance two and is definite. If we add $n$ after $(n-1)$ and before $(n-2)$ then the cycles
$(1(n-1)2 \ldots  n \ldots (n-3)(n-2)),(1(n-1)n2 \ldots (n-3)(n-2))$ has diagonal distance two and are definite. The cycle $(12 \ldots (n-1) \ldots n \ldots (n-3)(n-2))$ does not have diagonal distance two and is indefinite because it has $(12(n-1)n(n-3)(n-2))$ as a sub-cycle which satisfies the criterion of Theorem~\ref{theorem:DefIndefRegions}(B).  
\item $(n-1),(n-3)$ appear consecutively in this order and $(n-2)$ appears after $(n-3)$.
Now consider the cycle 
\equ{\big(123\ldots(n-1)(n-3)(n-2)\big).}
If we add $n$ after $(n-3)$ then the resulting $2$-standard cycle has diagonal distance two and is definite. If we add $n$ in between $(n-1)$ and $(n-3)$ then the resulting $2$-standard cycle does not have diagonal distance two and is indefinite because it has $(12(n-1)n(n-3)(n-2))$ as a sub-cycle which satisfies the criterion of Theorem~\ref{theorem:DefIndefRegions}(B). 
\item $(n-1),(n-3)$ appear consecutively in this order and $(n-2)$ appears before $(n-3)$.
Now consider the cycle 
\equ{\big(1\ldots (n-2)\ldots(n-1)(n-3)\big).}
If we add $n$ after $(n-3)$ then the resulting $2$-standard cycle has diagonal distance two and is definite. If we add $n$ in between $(n-1)$ and $(n-3)$ then the cycles $\big(1\ldots (n-5)(n-4)(n-2)(n-1)n(n-3)\big)$ and $\big(1\ldots (n-5)(n-2)(n-4)(n-1)n(n-3)\big)$ has diagonal distance two and are definite. The cycle $\big(1\ldots (n-2) \ldots (n-5)(n-4)(n-1)n(n-3)\big)$ does not have diagonal distance two and is indefinite because it has $((n-5)(n-4)(n-1)n(n-3)(n-2))$ as a sub-cycle which satisfies the criterion of Theorem~\ref{theorem:DefIndefRegions}(B).
\item $(n-3),(n-1)$ appear consecutively in this order and $(n-2)$ appears before $(n-3)$. Now consider the cycle
\equ{\big(1\ldots (n-2)\ldots(n-3)(n-1)\big).} 
Now $n$ has to be added at the end and the resulting $2$-standard cycle has diagonal distance two and is definite.
\item $(n-3),(n-1)$ appear consecutively in this order and $(n-2)$ appears after $(n-3)$.
Now consider the cycle 
\equ{\big(1\ldots (n-1)(n-3)(n-2)\big).}
If $n$ is added after $(n-3)$ then the resulting $2$-standard cycle has diagonal distance two and is definite. If $n$ is added in between $(n-1)$ and $n-3$ then the resulting $2$-standard cycle does not have diagonal distance two and is indefinite because it has $(12(n-1)n(n-3)(n-2))$ as a sub-cycle which satisfies the criterion of Theorem~\ref{theorem:DefIndefRegions}(B).
\item $1,3$ appear consecutive with $1$ first and $3$ just next to it.
Now consider the cycle 
\equ{\big(13\ldots 2 \ldots (n-1)\big),\big(13\ldots (n-1)2\big).}   
Now $n$ must be added at the end in $\big(13\ldots 2 \ldots (n-1)\big)$ and the resulting $2$-standard cycle has diagonal distance two and is definite. In $\big(13\ldots (n-1)2\big)$, $n$ must be added after $(n-1)$ and the resulting $2$-standard cycle has diagonal distance two and is definite.
\item $1,3$ appear consecutive with $1$ first and $3$ at the end.
Now consider the cycle 
\equ{\big(1\ldots 2 \ldots (n-1)3\big).} 
If $n$ is added in between $(n-1)$ and $3$ then the resulting $2$-standard cycle has diagonal distance two and is definite. If $n$ is added after $3$ at the end then the cycles $(1425\ldots (n-1)3n),(1245\ldots (n-1)3n)$ have diagonal distance two and are definite. The cycle $(145 \ldots 2\ldots (n-1)3n)$ does not have diagonal distance two and is indefinite because it has $(23n145)$ as a sub-cycle which satisfies the criterion of Theorem~\ref{theorem:DefIndefRegions}(B).

\end{enumerate}
If a $2$-standard consecutive $(n-1)\operatorname{-}$cycle $(1=a_1a_2\ldots a_{n-1})$ is already indefinite satisfying the criterion of Theorem~\ref{theorem:DefIndefRegions}(B) then it satisfies the criterion of Theorem~\ref{theorem:DefIndefRegions}(B) after adding $n$ and obtain a $2$-standard consecutive $n\operatorname{-}$cycle. We will now show that it does not have diagonal distance two.
Let $(1=a_1\ldots a_{n-1})$ be an indefinite $(n-1)\operatorname{-}$cycle. By induction it does not have diagonal distance two. After adding $n$ we show that $n$ cannot be adjacent to $2$ or $(n-2)$. Suppose the $(n-1)\operatorname{-}$cycle is given by $(1=a_1\ldots (n-2) \ldots (n-1) \ldots a_{n-1})$. Then $n$ should appear after $(n-1)$ unless the cycle is $(12\ldots (n-3)n(n-2)(n-1))$ or $(12\ldots (n-2)n(n-1))$, which is impossible. Hence $n$ and $(n-2)$ cannot be adjacent. 
Suppose $(n-2)$ appears after $(n-1)$ in the indefinite $(n-1)\operatorname{-}$cycle $(1=a_1\ldots a_{n-1})$. Then it is given by $(1=a_1\ldots (n-1) \ldots (n-2)=a_{n-1})$, which has diagonal distance two. Hence a contradiction. By a similar reasoning $n$ cannot appear adjacent to $2$. So the resulting $2$-standard consecutive $n\operatorname{-}$cycle does not have diagonal distance $2$.
This completes the proof of main Theorem~\ref{theorem:DefIndefRegions}.
\end{proof}
\section{\bf{The Main Theorem in terms of Matroids}}
We state and prove the theorem.

\begin{theorem}
\label{theorem:Reducibility}
Let $n\geq 3$ be a positive integer.
\begin{enumerate}
\item 
The map $(\gp_e)_{\mid_{\mcl{R}_{gen}(D_n)}}:\mcl{R}_{gen}(D_n)\lra \mcl{R}_{gen}(C_n=D_n\bs p)$ induced by deletion of the point $p$ is surjective if and only if the $2$-standard consecutive cycle associated to the region/residence of $p=n+1$ is a definite cycle as given in Theorem~\ref{theorem:DefIndefRegions}(A).
\item The map $\gp_e:\mcl{R}(D_n)\lra \mcl{R}(C_n=D_n\bs p)$ is surjective if and only if the $2$-standard consecutive cycle associated to the region/residence of $p=n+1$ is a definite cycle as given in Theorem~\ref{theorem:DefIndefRegions}(A).
\end{enumerate} 
\end{theorem}
\begin{proof}
We prove $(1)$.
Clearly the region/cycle $R$ is definite as given in Theorem~\ref{theorem:DefIndefRegions}(A) if and only if it occurs in every convex $n$-gon with generic diagonals. Hence  $p$ is in such a region $R$ if and only if the map $(\gp_e)_{\mid_{\mcl{R}_{gen}(D_n)}}:\mcl{R}_{gen}(D_n)\lra \mcl{R}_{gen}(C_n=D_n\bs p)$ is surjective.

We prove $(2)$. Suppose $\gp_e:\mcl{R}(D_n)\lra \mcl{R}(C_n)$ is surjective. Then $\mcl{R}_{gen}(D_n)\sbq \gp_e^{-1}(\mcl{R}_{gen}(C_n))\sbq \mcl{R}(D_n)$ and $\gp_e: \gp_e^{-1}(\mcl{R}_{gen}(C_n))\lra \mcl{R}_{gen}(C_n)$ is surjective. Since the fibre is associated to a convex open set which is the residence of a point $p$, we can also choose $p$ in a further general position in the same convex open set. So we obtain that actually the following map  
\equ{(\gp_e)_{\mid_{\mcl{R}_{gen}(D_n)}}:\mcl{R}_{gen}(D_n)\lra \mcl{R}_{gen}(C_n=D_n\bs p)}
is surjective. Now we use $(1)$ to conclude that $R$ is a definite region and its associated cycle is a definite cycle as described in Theorem~\ref{theorem:DefIndefRegions}(A).

Conversely the regions corresponding to definite cycles as described in Theorem~\ref{theorem:DefIndefRegions}(A) appear in the class $\mcl{C}_n$ of all convex $n$-gons, not just the class $(\mcl{C}_n)_{gen}$  of convex $n$-gons with generic diagonals. Note that for $n=7$ the region corresponding to the definite cycle $(1526374)$ also occurs in the class of all convex $n$-gons. Hence $(2)$ follows.
\end{proof}

\end{document}